\documentclass[a4paper,reqno,11pt]{amsart}
\usepackage[margin=1.3in]{geometry}
\usepackage[utf8]{inputenc}	
\usepackage{indentfirst} 
\usepackage{amsmath,amsfonts,amssymb,,mathrsfs}
\usepackage{bbm}
\usepackage[colorlinks=true]{hyperref}
\usepackage{enumerate}
\usepackage{graphicx}
\usepackage{xcolor}

\newcommand{\restr}{%
  \,\raisebox{-.127ex}{\reflectbox{\rotatebox[origin=br]{-90}{$\lnot$}}}\,%
}

\newcommand{\UUU}{\color{black}}
\newcommand{\EEE}{\color{black}}

\theoremstyle{definition}
\newtheorem{definition}{Definition}[section]
\theoremstyle{remark}
\newtheorem{remark}[definition]{Remark}
\theoremstyle{theorem}
\newtheorem{theorem}[definition]{Theorem}
\theoremstyle{theorem}
\newtheorem{lemma}[definition]{Lemma}
\theoremstyle{remark}

\theoremstyle{theorem}

\theoremstyle{theorem}

\theoremstyle{theorem}

\numberwithin{equation}{section}
\allowdisplaybreaks

\def\XXint#1#2#3{{\setbox0=\hbox{$#1{#2#3}{\int}$}
      \vcenter{\hbox{$#2#3$}}\kern-.5\wd0}}

\definecolor{ao}{rgb}{0.0, 0.5, 0.0}

\newcommand{\Om}{\Omega}
\newcommand{\R}{\mathbb{R}}

\newcommand{\HH}{\mathcal{H}}
\newcommand{\di}{\mathrm{d}}
\newcommand{\M}{\mathbb{M}}

\title{A new proof of compactness in $G(S)BD$}
\date{\today}
\author[S. Almi]{Stefano Almi}
\address[Stefano Almi]{Faculty of Mathematics, University of Vienna, 
Oskar-Morgenstern-Platz 1, 1090 Wien, Austria.}
\email{stefano.almi@univie.ac.at}

\author[E. Tasso]{Emanuele Tasso}
\address[Emanuele Tasso]{Technische Universit\"at Dresden, Faculty of Mathematics, 01062 Dresden, Germany}
\email{emanuele.tasso@tu-dresden.de}


\begin{document}

\maketitle

\begin{abstract}
We prove a compactness result in $GBD$ which also provides a new proof of the compactness theorem in $GSBD$, due to Chambolle and Crismale~\cite[Theorem~1.1]{cris2}. Our proof is based on a Fr\'echet-Kolmogorov compactness criterion and does not rely on Korn or Poincar\'e-Korn inequalities.
\end{abstract}

\section{Introduction}

In this paper we prove a compactness result in $GBD$, which in particular provides an alternative proof of the compactness theorem in $GSBD$ obtained by Chambolle and Crismale in~\cite[Theorem~1.1]{cris2}. Referring to Section~\ref{s:preliminaries} for the notation used below, the theorem reads as follows.


\begin{theorem}
\label{t:comfk}
Let $U \subseteq \R^{n}$ be an open bounded subset of~$\R^{n}$ and let $u_k \in GBD(U)$ be such that 
\begin{equation}
\label{e:hp-mu}
\sup_{k\in \mathbb{N}} \, \hat{\mu}_{u_k}(U) < + \infty \,.
\end{equation}
Then, there exists a subsequence, still denoted by~$u_k$, such that the set 
\[
A := \{x \in U : \ |u_k (x)| \to +\infty \text{ as } k \to \infty  \}
\]
has finite perimeter, $u_k \to u$  a.e.~in~$U \setminus A$ for some function $u \in GBD(U)$ with $u=0$ in~$A$. Furthermore,
\begin{equation}
\label{e:jump-lsc}
\HH^{n-1}(\partial^{*}A) \leq \lim_{\sigma \to \infty}\liminf_{k \to \infty} \, \HH^{n-1}(J^\sigma_{u_{k}}) \,,
\end{equation} 
where $J^\sigma_{u_k}:= \{x \in J_{u_k} \ \UUU : \ |[u_{k} (x) ] | \EEE \geq \sigma\}$.
\end{theorem}

We notice that the main difference with~\cite{cris2} is that we do not request equi-integrability of the approximate symmetric gradient~$e(u_{k})$ and boundedness of the measure of the jump sets~$J_{u_{k}}$, but only boundedness of~$\hat{\mu}_{u_{k}}(U)$, which is the natural assumption for sequences in $GBD(U)$. Hence, when passing to the limit, the absolutely continuous and the singular parts of~$\hat{\mu}_{u_{k}}$ could interact. For this reason, it is not possible to get weak $L^{1}$-convergence of the approximate symmetric gradients or lower-semicontinuity of the measure of the jump.

Nevertheless, we are able to recover the lower-semicontinuity~\eqref{e:jump-lsc} for the set~$A$ where $|u_{k}| \to +\infty$. In particular, formula~\eqref{e:jump-lsc} highlights that the emergence of the singular set~$A$ results from an uncontrolled jump discontinuity along the sequence~$u_{k}$. Hence, an equi-boundedness of the measure of the super-level sets~$J_{u_{k}}^{\sigma}$, i.e.,
\begin{displaymath}
\text{for every $\varepsilon>0$ there exists $\sigma_{\varepsilon} \in \mathbb{N}$ such that $\HH^{n-1} (J^{\sigma}_{u_{k}}) <\varepsilon$ for $\sigma \geq \sigma_{\varepsilon}$ and $k \in \mathbb{N}$,} 
\end{displaymath}
guarantees $\partial^*A = \emptyset$.

The $GSBD$-result~\cite[Theorem~1.1]{cris2} is recovered by replacing~\eqref{e:hp-mu} with
\begin{equation}
\label{e:hp-add}
\sup_{k \in \mathbb{N}} \int_{U} \phi( | e(u_{k} ) | ) \, \di x + \HH^{n-1}(J_{u_{k}} ) <+\infty\,,
\end{equation}
for a positive function~$\phi$ with superlinear growth at infinity. The novelty of our proof, presented in Section~\ref{appendix}, concerns the compactness part of Theorem \ref{t:comfk}. It is based on the Fr\'echet-Kolmogorov criterion and makes no use of Korn or Korn-Poincar\'e type of inequalities~\cite{MR3549205} (see also~\cite{CCS20, MR3714634, fri}), which are instead the key tools of~\cite{cris2}. The remaining lower-semicontinuity results of~\cite[Theorem~1.1]{cris2} can be obtained by standard arguments.

\section{Preliminaries and notation}\label{s:preliminaries}

We briefly recall here the notation used throughout the paper. For $d, k \in \mathbb{N}$, we denote by~$\mathcal{L}^{d}$ and $\HH^{k}$ the Lebesgue and the $k$-dimensional Hausdorff measure in~$\R^{d}$, respectively. Given $F \subseteq \R^{d}$, we indicate with $\text{dim}_{\HH}(F)$ the Hausdorff dimension of~$F$. For every compact subsets $F_{1}$ and~$F_{2}$ of~$\R^{d}$, ${\rm dist}_{\HH}(F_{1}, F_{2})$ stands for the Hausdorff distance between~$F_{1}$ and~$F_{2}$. We denote by~$\mathbbm{1}_{E}$ the characteristic function of a set~$E \subseteq \R^{d}$. For every measurable set $\Om \subseteq \R^{d}$ and every measurable function $u \colon \Om \to \R^{d}$, we further set $J_{u}$ the set of approximate discontinuity points of~$u$ and
\begin{displaymath}
J^{\sigma}_{u} := \{ x \in J_{u} : | [u] (x) | \geq \sigma\} \qquad \sigma >0\,,
\end{displaymath}
where $[u](x) := u^{+}(x) - u^{-}(x)$, $u^{\pm}(x) $ being the unilateral approximate limit of~$u$ at~$x$.

For $m, \ell \in \mathbb{N}$ we denote by~$\mathbb{M}^{m \times \ell}$ the space of $m \times \ell$ matrices with real coefficients, and set $\mathbb{M}^{m} := \mathbb{M}^{m \times m}$. The symbol~$\mathbb{M}^{m}_{sym}$ (resp.~$\mathbb{M}^{m}_{skw}$) indicates the subspace of~$\mathbb{M}^{m}$ of squared symmetric (resp.~skew-symmetric) matrices of order~$m$. We further denote by $SO(m)$ the set of rotation matrices.  

Let us now fix $n \in \mathbb{N}\setminus\{0\}$. For every $\xi \in \mathbb{S}^{n-1}$,~$\pi_{\xi}$ stands for the projection over the subspace~$\xi^{\perp}$ orthogonal to~$\xi$.  For every measurable set $V \subseteq \R^{n}$, every  $\xi \in \mathbb{S}^{n-1}$, and every $y \in \R^{n}$, we set 
\begin{align*}
& \Pi^{\xi} := \{ z \in \R^{n}: z \cdot \xi = 0\}\,, \qquad V^{\xi}_{y} := \{t \in \R: y + t\xi \in V\}\,.
\end{align*}  
 For $V \subseteq \R^{n}$ measurable,  $\xi \in \mathbb{S}^{n-1}$, and $y \in \R^{n}$ we define
\begin{align*}
& \UUU \hat{u}^{\xi}_{y} (t) \EEE := u(y + t\xi) \cdot \xi \qquad \text{for every $t \in V^{\xi}_{y}$}\,.
\end{align*}

For every open bounded subset $U$ of~$\R^{n}$, the space $GBD(U)$ of generalized functions of bounded deformation~\cite{dal} is defined as the set of measurable functions $u \colon U \to \R^{n}$ which \UUU admit \EEE a positive Radon measure $\lambda \in \mathcal{M}^{+}_{b}(U)$ such that for every $\xi \in \mathbb{S}^{n-1}$ one of the two equivalent conditions is satisfied \cite[Theorem~3.5]{dal}:
\begin{itemize}

\item for every $\theta \in C^{1}(\R; [-\tfrac{1}{2}; \tfrac{1}{2}])$ such that $0 \leq \theta' \leq 1$, the partial derivative $D_{\xi} (\theta (u \cdot \xi))$ is a Radon measure in $U$ and $| D_{\xi} (\theta (u \cdot \xi)) | (B) \leq \lambda(B)$ for every Borel subset~$B$ of~$U$;

\item for $\HH^{n-1}$-a.e.~$y \in \Pi_{\xi}$ the function~$\hat{u}^{\xi}_{y}$ belongs to $BV_{loc}(U^{\xi}_{y})$ and
\begin{equation}\label{e:GSBD}
\int_{\Pi^{\xi}} \big| (D \hat{u}^{\xi}_{y}) \big| \big( B^{\xi}_{y} \setminus J^{1}_{\hat{u}^{\xi}_{y}} \big) + \HH^{0} \big( B^{\xi}_{y} \cap J^{1}_{\hat{u}^{\xi}_{y}} \big) \, \di \HH^{n-1}(y) \leq \lambda(B)
\end{equation}
for every Borel subset~$B$ of~$U$.

\end{itemize}
A function $u$ belongs to $GSBD(U)$ if $\hat{u}^{\xi}_{y} \in SBV_{loc}(U^{\xi}_{y})$ and~\eqref{e:GSBD} holds. Every function $u \in GBD(U)$ admits an approximate symmetric gradient~$e(u) \in L^{1}(U ; \mathbb{M}^{n}_{sym})$. The jump set~$J_u$ is countably $(\HH^{n-1}, n-1)$-rectifiable with approximate unit normal vector~$\nu_{u}$. We will also use measures $\hat{\mu}^{\xi}, \hat{\mu}_{u} \in \mathcal{M}^{+}_{b}(U)$ defined in~\cite[Definitions~4.10 and~4.16]{dal} for $u \in GBD(U)$ and $\xi \in \mathbb{S}^{n-1}$. We further refer to~\cite{dal} for an exhaustive discussion on the fine properties of functions in~$GBD(U)$.

\section{Proof of Theorem~\ref{t:comfk}}
\label{appendix}

This section is devoted to the presentation of an alternative proof of Theorem~\ref{t:comfk}, based on the Fr\'echet-Kolmogorov compactness criterion.
We start by giving two definitions.

 \begin{definition}\label{d:3.1}
Let $\Xi = \{ \xi_{1}, \ldots, \xi_{n} \}$ denote an orthonormal basis of~$\mathbb{R}^n$. We define 
 \[
 S_{\Xi,0} := \bigcup_{\xi \in \Xi}\{x \in \mathbb{R}^n : \, |x|=1, \ x \in \Pi^{\xi}\}.
 \]
  Given $\delta >0$ we define the $\delta$-neighborhood of $S_{\Xi,0}$ as 
 \[
 S_{\Xi,\delta} := \{x \in \mathbb{R}^n : \,  |x|=1, \ \ \text{dist}(x,S_{\Xi,0}) < \delta \}.
 \]
 \end{definition}

 \begin{definition}
 \label{d:fambas}
 In order to simplify the notation, given a family $\mathcal{K}$ and a positive natural number $m$, we denote by $\mathcal{K}_m$ the set consisting of all subsets of $\mathcal{K}$ containing exactly $m$-elements of $\mathcal{K}$, i.e.
 \[
 \mathcal{K}_m := \{\mathcal{Z} \in \text{P}(\mathcal{K}) : \,  \#\mathcal{Z} = m\}.
 \]
 \end{definition}

 In order to prove Theorem~\ref{t:comfk}, we need the following two lemmas, which allow us to construct a suitable orthonormal basis of~$\R^{n}$ that will be used to test the Fr\'echet-Kolmogorov compactness criterium.

 \begin{lemma}
 \label{skeint}
Let $M \in \mathbb{N}$ be such that $M \geq n$ and consider a family $\mathcal{K} := \{\Xi_1, \dotsc,\Xi_M \}$ of orthonormal bases of $\mathbb{R}^n$ such that for every $\mathcal{Z} \in \mathcal{K}_{n}$
 \begin{align}
 \label{e:skint1}
 \bigcap_{\Xi \in \mathcal{Z}} S_{\Xi,0} = \emptyset\,.   
 \end{align}
 Then, there exists a further orthonormal basis $\Sigma = \{ \xi_1, \dotsc, \xi_n\} $ such that for every $\mathcal{Z} \in \mathcal{K}_{n-1}$
  \begin{align}
 \label{e:skint2}
 S_{\Sigma,0} \cap \bigcap_{\Xi \in \mathcal{Z}} S_{\Xi,0} = \emptyset \,.
 \end{align}
 \end{lemma}
 
 \begin{proof}
First of all notice that whenever $\mathcal{Z} \in \mathcal{K}_n$ is such that
 \[
 \bigcap_{\Xi \in \mathcal{Z}} S_{\Xi,0} = \emptyset
 \]
 then we have
 \begin{equation}
     \label{e:skint3}
     \mathcal{H}^0 \Big(\bigcap_{\Xi \in \mathcal{X}} S_{\Xi,0} \Big) < + \infty \qquad \text{for every $\mathcal{X} \in \mathcal{Z}_{n-1}$}.
 \end{equation}
 Indeed, let us suppose by contradiction that \eqref{e:skint3} does not hold for some $\mathcal{X} \in \mathcal{Z}_{n-1}$. Since for $\Xi \in \mathcal{X}$ we have that each $S_{\Xi,0}$ is a finite union of $(n-1)$-dimensional subspaces of~$\mathbb{R}^n$ intersected with~$\mathbb{S}^{n-1}$, the equality $\mathcal{H}^0 \Big(\bigcap_{\Xi \in \mathcal{X}} S_{\Xi,0} \Big)=+\infty$ implies that
 \[
 \text{dim}_{\mathcal{H}}\, \Big(\bigcap_{\Xi \in \mathcal{X}} S_{\Xi,0} \Big) \geq 1 \,.
 \]
 As a consequence we get
 \[
 \text{dim}_{\mathcal{H}}\, \Big(\bigcap_{\Xi \in \mathcal{X}} \bigcup_{\xi \in \Xi}\{ \xi^\bot  \} \Big) \geq 2 \,.
 \]
 Hence, if we denote by $\overline{\Xi}$ the basis contained in $\mathcal{Z} \setminus \mathcal{X}$, then by using Grassmann's formula
 \[
 \text{dim}(V) + \text{dim}(W) - \text{dim}(V \cap W) = \text{dim}(V+W)\leq n \, ,
 \]
 which is valid for each couple $V,W$ of vector subspaces of $\mathbb{R}^n$, we deduce
 \[
 \text{dim}_{\mathcal{H}}\, \Big( \bigcup_{\xi \in \overline{\Xi}}\{\xi^\bot\} \cap \bigcap_{\Xi \in \mathcal{X}} \bigcup_{\xi \in \Xi}\{ \xi^\bot  \} \Big) \geq 1 \,,
 \]
 hence
 \[
 \bigcap_{\Xi \in \mathcal{Z}} S_{\Xi,0} \neq \emptyset
 \]
 which is a contradiction to the assumption \eqref{e:skint1}.
 
  Fix $\{e_1, \dotsc, e_n \}$ an orthonormal basis of~$\mathbb{R}^n$ and let $SO(n)$ be the group of special orthogonal matrices. It can be endowed with the structure of an $\big(\frac{n^2-n}{2}\big)$-dimensional submanifold of $\mathbb{R}^{n^{2}}$. We can identify an element $O \in SO(n)$ with an $(n \times n)$-matrix whose columns are the vectors of an orthonormal basis~$\Xi$ written with respect to $\{e_1, \dotsc, e_n  \}$ and viceversa. 
  
In order to show the existence of~$\Sigma$ satisfying~\eqref{e:skint2} we prove the following stronger  condition:  given $\mathcal{Z} \in \mathcal{K}_{n-1}$, for $\mathcal{H}^{(n^2-n)/2}$-a.e. choice of~$\Sigma$ we have that
 \begin{equation}
 \label{e:empty}
 S_{\Sigma,0} \cap \bigcap_{\Xi \in \mathcal{Z}} S_{\Xi,0}  =\emptyset \, .
 \end{equation}
 This easily implies the existence of an orthonormal basis~$\Sigma$ satisfying~\eqref{e:skint2}, as the choice of~$\mathcal{Z} \in \mathcal{K}_{n-1}$ is finite. To show~\eqref{e:empty}, for every $i \in \{1,\dotsc, n\}$  let us define the smooth map $\Lambda_{i} \colon SO(n) \times \{y \in \mathbb{R}^{n-1} : \,  |y|=1 \} \to \mathbb{S}^{n-1}$ as
 \[
\Lambda_{i}(\Sigma , y) := \sum_{ j < i} y_j\xi_j + \sum_{j>i} y_{j-1} \xi_{j} \,,
 \]
 where~$\xi_j$ denotes the $j$-th column vector of the matrix representing~$\Sigma$. In order to show~\eqref{e:empty}, we claim that it is enough to prove that for every $x \in \mathbb{S}^{n-1}$ we have
 \begin{equation}
 \label{e:thin}
 \mathcal{H}^{(n^2-n )/2} \big( \pi_{SO(n)}(\{\Lambda_{i}^{-1}(x)\}) \big) = 0 \qquad \text{for $i \in \{1, \dotsc, n\}$},
 \end{equation}
 where $\pi_{SO(n)} \colon SO(n) \times \{y \in \mathbb{R}^{n-1} : \,  |y|=1 \} \to SO(n)$ is the canonical projection map. Indeed, if~$\Sigma$ does not belong to $\pi_{SO(n)}(\{\Lambda_{i}^{-1}(x)\})$ for every $x \in \bigcap_{\Xi \in \mathcal{Z}} S_{\Xi,0}$ and for every $i \in \{1, \dotsc, n\}$, then by using the definition of the map~$\Lambda_i$ we deduce immediately that~$\Sigma$ satisfies $ S_{\Sigma,0} \cap \bigcap_{\Xi \in \mathcal{Z}} S_{\Xi,0}  =\emptyset$. Therefore, if~\eqref{e:thin} holds, then the set (remember that $\bigcap_{\Xi \in \mathcal{Z}} S_{\Xi,0}$ is a discrete set)
 \[
 \bigcup_{i=1}^n \bigcup_{x \in \bigcap_{\Xi \in \mathcal{Z}} S_{\Xi,0} }\pi_{SO(n)}(\{\Lambda_{i}^{-1}(x)\})
 \] 
 is of $\mathcal{H}^{(n^2-n)/2}$-measure zero and~\eqref{e:empty} holds true. Thus, $\HH^{(n^{2} - n)/2}$-a.e.~$\Sigma$ satisfies~\eqref{e:skint2}.
 
 To prove~\eqref{e:thin} it is enough to show that the differential of $\Lambda_{i}$ has full rank at every point $z \in SO(n) \times \{y \in \mathbb{R}^{n-1} : \, |y|=1 \}$. Indeed, this implies that~$\Lambda_{i}^{-1}(x)$ is an $\big(\frac{n^2-n -2}{2}\big)$-dimensional submanifold for every $x \in \mathbb{S}^{n-1}$, which ensures the validity of~\eqref{e:thin} since 
 \begin{align*}
 & \#\big(\{\pi_{SO(n)}^{-1}(\Xi)\} \cap \{\Lambda_i^{-1}(x)\}\big)=1, \ \ x \in \mathbb{S}^{n-1}, \\ 
& \frac{n^2-n-2}{2} < \frac{n^2-n}{2} = \text{dim}_{\mathcal{H}}(SO(n)) \ \ (n \geq 2)\,.
 \end{align*}
 Notice that $\Lambda_{i}$ is the restriction to $SO(n) \times \{y \in \mathbb{R}^{n-1} : \,  |y|=1 \}$ of the map $\tilde{\Lambda}_{i} \colon \mathbb{M}^{n} \times \mathbb{R}^{n-1} \to \mathbb{R}^n$ defined as
 \[
\tilde{\Lambda}_{i}(\Theta,y) := \sum_{j< i}y_j \theta_j + \sum_{j>i} y_{j-1}  \theta_{j}\,, 
 \]
 where~~$\theta_j$ is the $j$-th column vector of the matrix $\Theta \in \mathbb{M}^{n}$. To show that the differential of $\Lambda_{i}$ has full rank everywhere, it is enough to check that for every $z \in SO(n) \times \{y \in \mathbb{R}^{n-1} : \,  |y|=1 \}$ the differential of $\tilde{\Lambda}_{i}$ restricted to $\text{Tan}(SO(n) \times \{y \in \mathbb{R}^{n-1} :\,  |y|=1 \},z)$ has rank equal to $n-1$. By using the relation 
 \[
\tilde \Lambda_i (M\Theta,y)=M  \tilde \Lambda_i(\Theta,y) \,,
 \]
 valid for every $M \in \mathbb{M}^n$, we can reduce ourselves to the case $z = (\mathrm{I} , \overline{y} )$, where~$\mathrm{I}$ denotes the identity matrix and $\overline{y} \in \mathbb{R}^{n-1}$ is such that $|\overline{y}|=1$. It is well known that
 \[
 \text{Tan}(SO(n) \times \{\zeta \in \mathbb{R}^{n-1} : \, |\zeta |=1 \},z) \cong \mathbb{M}_{skw}^{n} \times \text{Tan}(\{ \zeta  \in \mathbb{R}^{n-1} :\, | \zeta |=1 \}, \overline{y}) \,,
 \]
 where $\mathbb{M}_{skw}^{n}$ denotes the space of skew symmetric matrices. Using that $\mathbb{R}^{n^2+n - 1} \cong \mathbb{M}^{n} \times \mathbb{R}^{n-1} $, we identify a point $Z \in \mathbb{R}^{n^2+n-1}$ as $Z= ((x^i_j)_{i,j=1}^n,y_1,\dotsc,y_{n-1})$. A direct computation shows that the differential of~$\Lambda_{i}$ at the point $(\mathrm{I},\overline{y})$ acting on the vector~$Z$ is given by 
 \[
d \tilde{\Lambda}_{i}(\mathrm{I} ,\overline{y})[Z] = \sum_{l=1}^n \sum_{j < i}  (x^j_l\overline{y}_j+ \delta_{jl} y_j)e_l + \sum_{j>i} ( x^{j}_{l} \overline{y}_{j-1} + \delta_{jl} y_{j-1}) e_{l}\,.
 \]

 It is better to introduce the matrix $P_i \in \mathbb{M}^{n\times (n-1)}$ defined as 
\[
(P_i)_k^m:= 
\begin{cases}
 \delta_{km} &\text{ if }1\leq m<i \,,\\
 \delta_{k-1m} &\text{ if }i \leq m \leq n-1\,.
\end{cases}
\]
Roughly speaking, given $X \in \mathbb{M}^{l \times n}$, the product $XP_i$ is the matrix in $\mathbb{M}^{l\times (n-1)}$ obtained by removing from~$X$ the i-th column, while given $Y \in \mathbb{M}^{(n-1) \times l}$, the product~$P_iY$ is the matrix in~$\mathbb{M}^{n \times l}$ obtained by adding a new row made of zero entries at the $i$-th position. With this definition the linear map $d\Lambda_i(\mathrm{I},\overline{y})(\cdot)$ can be rewritten more compactly as
\begin{equation}
\label{e:comsys1}
d\Lambda_i(\mathrm{I},\overline{y})[(X,y)]=XP_i\overline{y}+P_iy, \ \ X \in \mathbb{M}_{skw}^n, \ \ y \in \text{Tan}(\{ \zeta \in \mathbb{R}^{n-1}: \, | \zeta |=1  \},\overline{y})\,.
\end{equation}
Given $O \in SO(n-1)$ such that $O\tilde{e}_1=\overline{y}$ ($\{\tilde{e}_1,\dotsc,\tilde{e}_{n-1} \}$ denotes the reference orthonormal basis of $\mathbb{R}^{n-1}$), we can rewrite the system as
\begin{equation}
    \label{e:comsys2}
    d\Lambda_i(\mathrm{I},\overline{y})[(X,y)]=XP_iO \tilde{e}_1 +P_iy, \ \ X \in \mathbb{M}_{skw}^n, \ \ y \in \text{Tan}(\{\zeta  \in \mathbb{R}^{n-1} :  | \zeta |=1  \},\overline{y}).
\end{equation}
Hence, by the well known relation
\begin{equation}
\label{e:imker}
\text{dim}(V) -\text{dim}(\text{Im}[\alpha]) = \text{dim}(\text{ker}[\alpha]) \,, 
\end{equation}
valid for every linear map $\alpha \colon V \to W$ and every finite dimensional vector spaces $V$ and $W$,
if we want to prove that $d\Lambda_i(\mathrm{I},\overline{y})$ has full rank, i.e. 
\begin{equation}
    \label{e:dim}
    \text{dim}(\text{Im}[(\cdot)P_iO\tilde{e}_1 + P_i(\cdot)]) = n-1 \,,
\end{equation}
since 
\[
n-1 \geq \text{dim}(\text{Im}[(\cdot)P_iO\tilde{e}_1 + P_i(\cdot)]) \geq \text{dim}(\text{Im}[(\cdot)P_iO\tilde{e}_1])
\]
(where the first inequality comes from $\text{Im}[d\Lambda_i(\mathrm{I},\overline{y})] \subset \text{Tan}(\mathbb{S}^{n-1},\Lambda_i(\mathrm{I},\overline{y}))$), it is enough to show that
\begin{equation}
\label{e:dimim}
\text{dim}(\text{Im}[(\cdot)P_iO\tilde{e}_1]) = n-1 \,.
\end{equation}
Again by relation \eqref{e:imker} we can reduce ourselves to find the dimension of the kernel of the map $\mathbb{M}^n_{skw} \ni X \mapsto XP_iO\tilde{e}_1$. But this dimension can be easily computed to be
\[
\text{dim}(\text{ker}[(\cdot)P_iO\tilde{e}_1]) = \sum_{k=1}^{n-2} k = \frac{(n-2)(n-1)}{2} \,,
\]
which immediately implies \eqref{e:dimim}.
 \end{proof}
 
 \begin{remark}
 \label{r:skecon}
By a standard argument from linear algebra it is possible to construct $n$ orthonormal bases of $\mathbb{R}^n$, say $\mathcal{K} = \{\Xi_1,\dotsc, \Xi_n \}$ satisfying
 \[
 \bigcap_{\Xi \in \mathcal{K}} S_{\Xi,0} = \emptyset \,.
 \]
 Moreover, given $U \subset SO(n)$ open, then $\Xi_i$ can be chosen in such a way that 
 \[
 \Xi_i \in U, \ \ i \in \{1,\dotsc,n \} \,.
 \]
 Therefore, Lemma \ref{skeint}, and in particular condition \eqref{e:empty}, tells us that for every $M \in \mathbb{N}$ ($M\geq n$) we can always find a family of orthonormal bases of $\mathbb{R}^n$, say $\mathcal{K} = \{\Xi_1,\dotsc, \Xi_M\}$, satisfying \eqref{e:skint1} and 
 \[
 \Xi_i \in U, \ \ i \in \{1,\dotsc,M\} \,.
 \]
  \end{remark}

 \begin{lemma}
 \label{ske}
 Let $A \subset \mathbb{R}^n$ be a measurable set with $\mathcal{L}^n(A) < \infty$, let $(B_k)_{k=1}^\infty$ be measurable subsets of $A$, and let $(v_k)_{k=1}^\infty$ be measurable functions $v_k \colon B_k \to \mathbb{S}^{n-1}$.  Then, given a sequence $\epsilon_h \searrow 0$, there exist a sequence $\delta_h \searrow 0$ with $\delta_h >0$, a map $\phi \colon \mathbb{N} \to \mathbb{N}$, and an orthonormal basis~$\Xi$ of $\mathbb{R}^n$ such that, up to passing through a subsequence on~$k$, $\mathcal{L}^n(v_k^{-1}(S_{\Xi,\delta_h})) \leq \epsilon_h$ for every $k \geq \phi(h)$. 
 \end{lemma}
 
 \begin{proof}
We claim that for every natural number $N \geq n$, for every $j \in \{0,1,\dotsc,n-1\}$, for every $\varepsilon >0$, and for every open set $U \subset SO(n)$ there exist $\delta >0$ and a family of orthonormal bases $\mathcal{K} := \{\Xi_1,\dotsc,\Xi_N \} \subseteq U$, such that, up to subsequences on $k$, we have 
 \begin{align}
     \label{e:lA41}
     \mathcal{L}^n \Big(v_k^{-1} \Big(\{ x \in \bigcap_{\Xi \in \mathcal{Z}} S_{\Xi,\delta}  :& \, \mathcal{Z} \in \mathcal{K}_{n-j}\}\Big) \Big) \leq \varepsilon,  \ \ k=1,2,\dotsc, \\
     \label{e:lA411}
     \Xi \in U&, \ \ \Xi \in \mathcal{K} \,.
     \end{align}
 Clearly the pair $(\delta,\mathcal{K})$ depends on $(N,j,\varepsilon)$, but we do not emphasize this fact. We proceed by induction on $j$. The case $j=0$: given $N \in  \mathbb{N}$, $\varepsilon >0$, and any open set $U \subset SO(n)$, we can make use of Lemma \ref{skeint} and Remark \ref{r:skecon} to find $N$ orthonormal bases $\mathcal{K} =\{ \Xi_1, \dotsc, \Xi_N\} \subseteq U$ such that 
 \[ 
 \bigcap_{\Xi \in \mathcal{Z}} S_{\Xi,0} =\emptyset \qquad \text{for $\mathcal{Z} \in \mathcal{K}_n$}\,.
 \]
 Being $S_{\Xi,0}$ closed sets, there exists $\delta >0$ such that
 \[
 \bigcap_{\Xi \in \mathcal{Z}} S_{\Xi,\delta} =\emptyset \qquad \text{for $\mathcal{Z} \in \mathcal{K}_n$}.
 \]
 Hence, \eqref{e:lA41} is satisfied with $j=0$ and \eqref{e:lA411} holds true.
 
We want to prove the same for $0<j\leq n-1$. For this purpose we fix a natural number $M \geq n$, a parameter $\varepsilon>0$, and an open set $U \subset SO(n)$. By using the induction hypothesis, we may suppose that \eqref{e:lA41} and \eqref{e:lA411} hold true for $j-1$. This means that given $N \geq n$, $\tilde{\varepsilon}>0$ (to be chosen later), we find $\delta > 0$ and orthonormal bases $\mathcal{K}=\{\Xi_1,\dotsc, \Xi_N\}$ such that \eqref{e:lA41} and \eqref{e:lA411} hold true for $j-1$. Choose  $\mathcal{Z} \in \mathcal{K}_M$ and consider the following set 
\begin{equation}
 \label{e:lA48}
 S^{n-j}_{\mathcal{Z},\delta} := \bigcup_{q \in \mathcal{Z}_{n-j}} \bigcap_{\Xi \in q} S_{\Xi,\delta} \,.
 \end{equation}
 which is the union of all the possible $(n-j)$-intersections of sets of the form $S_{\Xi,\delta}$ for $\Xi \in \mathcal{Z}$. 
 
 We recall the following identity valid for any finite family of subsets of $A$, say $(B)_{l=1}^L$, which reads as
 \begin{equation}
     \label{e:lA44}
     \mathcal{L}^n \Big (\bigcup_{l=1}^L B_l \Big) = \sum_{l=1}^L \mathcal{L}^n (B_l)- \sum_{l_1 < l_2}^L \mathcal{L}^n(B_{l_1} \cap B_{l_2}) +  \dotsc +(-1)^{L-1} \mathcal{L}^n \Big (\bigcap_{l=1}^L B_l \Big) \,.
 \end{equation}
Now we partition $\mathcal{K}$ into $N/M$ disjoint subsets (without loss of generality we may choose $N$ to be an integer multiple of $M$) each of which belongs to $\mathcal{K}_M$. We call this partition $\mathcal{P}$. By construction, any $l$-intersection of sets of the form $S^{n-j}_{\mathcal{Z},\delta}$ with $\mathcal{Z} \in \mathcal{P}$ can be written as the union of ${M \choose n- j}^l$ sets each of which, thanks to the fact that (we use that $\mathcal{P}$ is a partition)
\[
Z_1,Z_2 \in \mathcal{P} \Rightarrow Z_1 \cap Z_2 = \emptyset \,,
\] 
is the intersection of at least $n-(j -1)$ different sets of the form $S_{\Xi,\delta}$ with $\Xi \in \mathcal{K}$. Taking this last fact into account, if we replace the sets $B_j$ with $v_k^{-1}(S^{n-j}_{\mathcal{Z},\delta})$ and $L = N/ M$ in identity \eqref{e:lA44}, we obtain 
\begin{equation}
\label{e:lA45}
\mathcal{L}^n \Big(\bigcup_{\mathcal{Z} \in \mathcal{P}} v_k^{-1}(S^{n-j}_{\mathcal{Z},\delta}) \Big) \geq \sum_{\mathcal{Z} \in \mathcal{P}} \mathcal{L}^n( v_k^{-1}(S^{n-j}_{\mathcal{Z},\delta})) - \sum_{l=2}^{N / M} {M \choose n-j}^l \tilde{\varepsilon}, \ \ k=1,2,\dotsc,
\end{equation}
 where we have used the inductive hypothesis \eqref{e:lA41} for $j-1$ to estimate the remaining terms in the right hand-side of \eqref{e:lA44}.
 
Now suppose that for every $\mathcal{Z} \in \mathcal{K}_M$ it holds true for some $k$  
 \begin{equation}
     \label{e:lA42}
     \mathcal{L}^n (v_k^{-1}(S^{n-j}_{\mathcal{Z},\delta})) > \varepsilon \,,
 \end{equation}
 then inequality \eqref{e:lA45} implies
 \begin{equation}
     \label{e:lA46}
     \mathcal{L}^n \Big(\bigcup_{\mathcal{Z} \in \mathcal{P}} v_k^{-1}(S_{\mathcal{Z},\delta}) \Big) > \frac{N}{M} \varepsilon - \sum_{l=2}^{N / M} {M \choose n-j}^l \tilde{\varepsilon} \,.
 \end{equation}
Therefore, if we choose $N$ sufficiently large in such a way that
 \[
 \frac{N}{M}\varepsilon \geq 2 \mathcal{L}^n(A) \,,
 \]
 and $\tilde{\varepsilon}>0$ such that
 \[
 \sum_{l=2}^{N / M} {M \choose n-j}^l \tilde{\varepsilon} < \mathcal{L}^n(A) \,,
 \]
 then \eqref{e:lA46} implies that for every $k$ there exists $\mathcal{Z}^k \in \mathcal{P}$ for which \eqref{e:lA42} does not hold, i.e.,
 \[
 \mathcal{L}^n(v_k^{-1}(S^{n-j}_{\mathcal{Z}^k,\delta})) \leq \varepsilon, \ \ k=1,2,,\dotsc,
 \]
 where we have used that $B_k$, the domain of $v_k$, is contained in $A$. Being $\mathcal{P}$ a finite family, we may suppose that, up to subsequences on $k$, we find a common $\mathcal{Z} \in \mathcal{P}$ for which  
 \begin{equation}
     \label{e:lA47}
     \mathcal{L}^n(v_k^{-1}(S^{n-j}_{\mathcal{Z},\delta})) \leq \varepsilon, \ \ k=1,2,,\dotsc.
 \end{equation}
Taking into account the definition of $S^{n-j}_{\mathcal{Z},\delta}$ \eqref{e:lA48}, formula \eqref{e:lA47} gives our claim for $j$. Finally, by induction, this implies the validity of our claim for every $j \in \{0,\dotsc,n \}$.

Now we prove the lemma. For $j=n-1$ the claim says in particular that we find an orthonormal basis $\Xi_0$ and $\delta_0 >0$ such that, up to pass to a subsequence on $k$, we have
\[
\mathcal{L}^n(v_k^{-1}(S_{\Xi_0,\delta_0})) \leq \epsilon_0, \ \ k=1,2,\dotsc .
\]
Notice that by using a continuity argument, we find a neighborhood $U_0$ of $\Xi_0$ in $SO(n)$ such that
\[
S_{\Xi,\delta_0/2} \Subset S_{\Xi_0,\delta}, \ \ \Xi \in U_0 \,.
\]
By applying again the claim we find an orthonormal basis $\Xi_1 \in U_0$ and $\tilde{\delta}_1 >0$ such that, up to pass to a further subsequence on $k$, we have
\[
\mathcal{L}^n(v_k^{-1}(S_{\Xi_1,\tilde{\delta}_1})) \leq \epsilon_1, \ \ k=1,2,\dotsc .
\]
Hence if we set $\delta_1 := \min\{\tilde{\delta}_1,\delta_0/2 \}$ we obtain as well
\begin{align*}
\mathcal{L}^n(v_k^{-1}(S_{\Xi_1,\delta_1})) &\leq \epsilon_1, \ \ k=1,2,\dotsc, \\
S_{\Xi_1,\delta_1} &\Subset S_{\Xi_0,\delta_0} \,.
\end{align*}
Proceeding again by induction, we find for every $h =1,2,\dotsc$ an orthonormal basis $\Xi_h$, $\delta_h >0$, and a subsequence $(k^h_{\ell})_{\ell}$, such that
\begin{align*}
\mathcal{L}^n(v_{k_{\ell}^h}^{-1}(S_{\Xi_h,\delta_h})) &\leq \epsilon_h, \ \ \ell=1,2,\dotsc, \\
S_{\Xi_h,\delta_h} &\Subset S_{\Xi_{h-1},\delta_{h-1}}\,, \\
(k^h_{\ell})_{\ell} &\subset (k^{h-1}_{\ell})_{\ell}\,.
\end{align*}
If we denote with abuse of notation the diagonal sequence $(k_h^h)_h$ simply as $k$, then we can find a map $\phi \colon \mathbb{N} \to \mathbb{N}$ such that
\begin{align}
\label{e:lA410}
\mathcal{L}^n(v_k^{-1}(S_{\Xi_h,\delta_h})) &\leq \epsilon_h, \ \ k \geq \phi(h) \\
\label{e:lA49}
S_{\Xi_h,\delta_h} &\Subset S_{\Xi_{h-1},\delta_{h-1}}\,.
\end{align}
Being the family $(S_{\Xi_h,0})_h$ made of compact subsets of $\mathbb{S}^{n-1}$, then it is relatively compact with respect to the Hausdorff distance. This means that, up to a subsequence on $h$, we find an orthonormal basis $\Xi$ such that
\[
\lim_{h \to \infty} \text{dist}_{\mathcal{H}}(S_{\Xi_h,0},S_{\Xi,0}) =0\,.
\]
By using \eqref{e:lA49} and the fact that $S_{\Xi_h,\delta_h}$ are relatively open subsets of $\mathbb{S}^{n-1}$, this last convergence tells us that for every $h$ the compact inclusion $S_{\Xi,0} \Subset S_{\Xi_h,\delta_h}$ holds true. But this implies that up to defining suitable $\delta'_h >0$ with $\delta'_h \leq \delta_h$, we can write
\[
S_{\Xi,\delta'_h} \Subset S_{\Xi_h,\delta_h}, \ \ h \in \mathbb{N}.
\]
Finally, with abuse of notation we set  $\delta_h := \delta'_h$ for every $h$. Then \eqref{e:lA410} implies 
\[
\mathcal{L}^n(v_k^{-1}(S_{\Xi,\delta_h})) \leq \epsilon_h, \ \ k \geq \phi(h), \ \ h \in \mathbb{N}.
\]
This gives the desired result.
 \end{proof}

\begin{remark}
\label{r:bound}
Given $U \subset \mathbb{R}^n$, $u \in GBD(U)$, and $\sigma \geq 1$, we have that  
\begin{equation}
\mathcal{H}^{n-1}(J^\sigma_u)\leq 4n\hat{\mu}_u(U) \,.
\end{equation}
Indeed, given $\epsilon>0$, one can consider a partition of $\mathbb{S}^{n-1}$ into a finite family of measurable sets $\{S_1,\dotsc,S_M\}$ such that for every $m=1,\dotsc,M$ there exists an orthonormal basis $\Xi_m = \{\xi^m_1,\dotsc,\xi^m_n\}$ with $\xi\cdot \xi^m_i \geq 1/4$ for every $\xi \in S_m$ and for every $i,j\in \{1,\dotsc,n\}$ and $m \in \{1,\dotsc,M\}$. Consider then the partition of $J^\sigma_u$ given by $\{B_1,\dotsc,B_M\}$ where $B_m:=\{x \in J^\sigma_u : \, [u(x)]/|[u(x)]| \in S_m\}$. We then have
\[
\begin{split}
\mathcal{H}^{n-1}(J^\sigma_u) &\leq \sum_{m=1}^M\sum_{\xi \in \Xi_m} \int_{B_m} |\nu_{u} \cdot \xi| \, \di \mathcal{H}^{n-1} = \sum_{m=1}^M\sum_{\xi \in \Xi_m}\int_{\Pi_{\xi}} \mathcal{H}^0( (B_m)_y^{\xi})\, \di \HH^{n-1}(y) \\
&= \sum_{m=1}^M\sum_{\xi \in \Xi_m}\int_{\Pi_{\xi}}  \mathcal{H}^0(J^1_{4\hat{u}_y^{\xi}} \cap (B_m)_y^{\xi})\, \di\HH^{n-1}(y) = \sum_{m=1}^M\sum_{\xi \in \Xi_m} \hat{\mu}^\xi_{4u}(B_m)\\
&\leq n\sum_{m=1}^M \hat{\mu}_{4u}(B_m)  \leq n\hat{\mu}_{4u}(U) \leq 4n \mu_{u}(U)\,,
\end{split}
\]
where we have used that $|[4\hat{u}_y^\xi](t)| \geq 1$ for every $t \in J_{4\hat{u}_y^\xi}\cap (B_m)_y^{\xi} $ for $\mathcal{H}^{n-1}$-a.e. $y \in \Pi^{\xi}$ with $\xi \in \Xi_m$.
\end{remark}
 
\begin{remark}
\label{r:sigma}
 Let $U \subset \mathbb{R}^n$ and $u \in GBD(U)$. Given $\xi \in \mathbb{S}^{n-1}$ and $\sigma>1$ if we introduce the map $\hat{\mu}^\xi_\sigma \colon \mathcal{B}(U) \to \overline{\mathbb{R}}$ as
 \begin{equation}
 \hat{\mu}^\xi_\sigma(B) := \int_{\Pi^\xi} |D\hat{u}_y^\xi|(B_y^\xi \setminus J_{\hat{u}_y^\xi}^\sigma) + \mathcal{H}^0(B_y^\xi \cap J_{\hat{u}_y^\xi}^\sigma) \, \di \HH^{n-1}(y), \ \ B \in \mathcal{B}(U)\,,
 \end{equation}
 then we have $\hat{\mu}^\xi_\sigma \in \mathcal{M}^+_b(U)$. More precisely, for $\mathcal{H}^{n-1}$-a.e. $y \in \Pi^\xi$ we have 
 \[
 \begin{split}
 &|D\hat{u}_y^\xi|(B \setminus J_{\hat{u}_y^\xi}^\sigma) + \mathcal{H}^0(B \cap J_{\hat{u}_y^\xi}^\sigma)\\
 &\leq |D\hat{u}_y^\xi|(B \setminus J_{\hat{u}_y^\xi}^1) + \mathcal{H}^0(B \cap J_{\hat{u}_y^\xi}^1) + (\sigma-1)\mathcal{H}^0(B \cap (J_{\hat{u}_y^\xi}^1 \setminus J_{\hat{u}_y^\xi}^\sigma)), \ \ B \in \mathcal{B}(U_y^\xi),
 \end{split}
 \]
 (notice that for $\mathcal{H}^{n-1}$-a.e.~$y$ the right hand side is a finite measure thanks to \UUU Remark~\ref{r:bound}). \EEE By using the inclusion $J^1_{\hat{v}^\xi_y} \subset (J^1_v)^\xi_y$, valid for every $v \in GBD(U)$ for every $\xi \in \mathbb{S}^{n-1}$ for $\mathcal{H}^{n-1}$-a.e.~$y \in \Pi^\xi$, we deduce
 \begin{equation}
 \label{e:3.23}
     \hat{\mu}^\xi_\sigma(B) \leq \hat{\mu}^\xi(B) + (\sigma-1) \int_{B \cap J^1_u} |\nu_{u} \cdot \xi| \, \di \mathcal{H}^{n-1}, \ \ B \in \mathcal{B}(U)\,.
 \end{equation}
 Finally, \UUU Remark \ref{r:bound} \EEE and the definition of~$\hat{\mu}^{\xi}$ (see~\cite[Definition~4.10]{dal}) imply that the right-hand side of~\eqref{e:3.23} is a finite measure, and so is~$\hat{\mu}^{\xi}_{\sigma}$.
\end{remark}

 We are now in a position to prove Theorem~\ref{t:comfk}.

\begin{proof}[Proof of Theorem~\ref{t:comfk}]
Let $\tau(t) := \arctan{(t)}$. We claim that for every $i\in \{1,\dotsc, n\}$ the family $(\tau(u_k \cdot e_i))_k$ is relatively compact in $L^1(U)$, where $\{e_i\}_{i=1}^{n}$ denotes a suitable orthonormal basis of~$\mathbb{R}^n$. Now given $\epsilon_h \searrow 0$, by using Lemma \ref{ske}, there exists $\delta_h \searrow 0$ such that if we define $B_k := \{|u_{k}|\neq 0\}$ and $v_{k} \colon B_k \to \mathbb{S}^{n-1}$ as $v_{k} := u_{k}/|u_{k}|$, then
\[
\mathcal{L}^n \big( v^{-1}_{k}(S_{\Xi,\delta_h}) \big) \leq \epsilon_h \qquad  \text{ for every }k \geq \phi(h)
\]
for a suitable orthonormal basis~$\Xi$ and a suitable map $\phi \colon \mathbb{N} \to \mathbb{N}$. 

In order to simplify the notation, let us denote $\Xi = \{ e_1, \dotsc, e_n\}$. Fix $i \in \{1,\dotsc,n\}$ and set  $\xi^t_j:= \frac{\sqrt{t}}{\sqrt{t+t^2}}e_i + \frac{t}{\sqrt{t + t^2}}e_j \in \mathbb{S}^{n-1}$ for every $j \neq i$ \UUU and $t >0$. \EEE Notice that
\begin{equation}
\label{asy}
|\xi^t_j -e_i| \leq \sqrt{2t} \qquad \text{and} \qquad \bigg|\frac{\xi^t_j -e_i}{|\xi^t_j -e_i|} - e_j \bigg| \leq \sqrt{2t}  \,.
\end{equation}
We define $U_t := \{x \in U :\, \text{dist}(\partial U , x) > t \}$. Since we want to apply Fr\'echet-Kolmogorov Theorem, we have to estimate for $x \in U_t$ 
\begin{equation*}\label{e:A1}
\begin{split}
|\tau(u_{k}(x + t e_j) \cdot e_i) & - \tau(u_{k}(x)\cdot e_i)| 
\\
&
\leq |\tau(u_{k}(x + t e_j) \cdot e_i ) - \tau(u_{k}(x + t e_j) \cdot \xi^t_j)| 
\\
&
\qquad + |\tau(u_{k}(x + t e_j) \cdot \xi^t_j) - \tau(u_{k}(x - \sqrt{t} e_i) \cdot \xi^t_j)|  
\\
&
\qquad + |\tau(u_{k}(x - \sqrt{t} e_i) \cdot \xi^t_j) -\tau(u_{k}(x - \sqrt{t} e_i) \cdot e_i) | 
\\
&
\qquad + |\tau(u_{k}(x - \sqrt{t} e_i) \cdot e_i) - \tau(u_{k}(x) \cdot e_i)| \,.
\end{split}
\end{equation*}

 Now notice that by definition of~$S_{\Xi, \delta_{h}}$ (see Definition~\ref{d:3.1}), there exists a positive constant $c = c(\delta_h)$ such that for every $x \in U \setminus v^{-1}_{k}(S_{\Xi,\delta_h/2})$ and every $i,j \in \{1,\dotsc, n\}$
\begin{equation}
\label{uniquo}
|u_{k}(x) \cdot e_i| \geq c(\delta_h) \, |u_{k}(x)\cdot e_j|  \qquad \text{for every } k \text{ and }h\,.
\end{equation}
Moreover, by taking into account \eqref{asy}, we deduce the existence of a dimensional parameter $\overline{t} >0$ such that
\begin{align}
   \label{eqsign1}
   |z \cdot \xi_j^t|^2 \geq 2^{-1}|z \cdot e_i|^2 &\qquad  t \leq \overline{t}, \  z \in \R^{n},  \ i,j \in \{1,\dotsc,n\} \\
    \label{eqsign2}
  \bigg|z \cdot \frac{\xi_j^t -e_i}{|\xi_j^t -e_i|}  \bigg| \leq 2 |z \cdot e_j| &\qquad   t \leq \overline{t}, \  z \in \R^{n},  \ i,j \in \{1,\dotsc,n\}.
\end{align}


For every $t \leq \overline{t}$, if $x \in U_t$ and $x \notin v^{-1}_{k}(S_{\Xi,\delta_h/2}) - te_j$, by using~\eqref{asy} and~\eqref{uniquo}-\eqref{eqsign2}, we can write
\begin{align}\label{e:A2}
& |\tau(u_{k}(x + t e_j) \cdot e_i ) - \tau(u_{k}(x + t e_j) \cdot \xi^t_j)| = \bigg|\int_{u_{k}(x+te_j)\cdot e_i}^{u_{k}(x+te_j)\cdot \xi_j^t} \frac{\di s}{1+ s^2}\bigg| 
\\
&
\leq \max\bigg\{\frac{\sqrt{2t}}{1+|u_{k}(x+te_j)\cdot e_i|^2},\frac{\sqrt{2t}}{1+|u_{k}(x+te_j)\cdot \xi_j^t|^2} \bigg\} \bigg|u_{k}(x+te_j)\cdot \frac{\xi_j^t - e_i}{|\xi_j^t - e_i|}\bigg| \nonumber 
\\
&
\leq \max\bigg\{\frac{\sqrt{2t}}{1+|u_{k}(x+te_j)\cdot e_i|^2},\frac{\sqrt{2t}}{1+2^{-1}|u_{k}(x+te_j)\cdot e_i|^2} \bigg\} \bigg|u_{k}(x+te_j)\cdot \frac{\xi_j^t - e_i}{|\xi_j^t - e_i|}\bigg| \nonumber 
\\
& 
\leq\frac{2\sqrt{2t}}{1+2^{-1}|u_{k}(x+te_j)\cdot e_i|^2}
  |u_{k}(x+te_j)\cdot e_j|
\leq\frac{2\sqrt{t}}{c(\delta_h)}
  \nonumber
\end{align}

and analogously if $x \in U_t$ and $x \notin v^{-1}_{k}(S_{\Xi,\delta_h/2}) + \sqrt{t}e_i$
\begin{equation}\label{e:A3}
|\tau(u_{k}(x - \sqrt{t} e_i) \cdot \xi^t_j) -\tau(u_{k}(x - \sqrt{t} e_i) \cdot e_i) | \leq \frac{2\sqrt{t}}{c(\delta_h)} \,.
\end{equation}

Hence, from~\eqref{e:A2} and~\eqref{e:A3} we infer that for every $t \leq \overline{t}$ 
 \[
 \int_{U_{t}} |\tau(u_{k}(x + t e_j) \cdot e_i ) - \tau(u_{k}(x + t e_j) \cdot \xi^t_j)| \, \di x \leq |U| \frac{2\sqrt{t}}{c(\delta_h)} + \UUU \pi \epsilon_h \,, \EEE
 \]
and 
\[
 \int_{U_{t}}  |\tau(u_{k}(x - \sqrt{t} e_i) \cdot e_i ) - \tau(u_{k}(x - \sqrt{t} e_i) \cdot \xi^t_j)| \, \di x \leq |U| \frac{2\sqrt{t}}{c(\delta_h)} + \UUU \pi \epsilon_h \,. \EEE
\]
 Moreover, setting $s_{t} := \sqrt{t + t^{2}}$ we can write
\begin{align}
\label{e:100}
  \int_{U_t} & |\tau(u_{k}(x + t e_j) \cdot \xi_j^t ) - \tau(u_{k}(x - \sqrt{t} e_i) \cdot \xi^t_j)| \, \di x\\
  &= \int_{U_t} |\tau(u_{k}(x - \sqrt{t} e_i + s_t \xi_j^t) \cdot \xi_j^t ) - \tau(u_{k}(x - \sqrt{t} e_i) \cdot \xi^t_j)| \, \di x \nonumber\\
  &= \int_{U_t + \sqrt{t}e_i} |\tau(u_{k}(x + s_t \xi_j^t) \cdot \xi_j^t ) - \tau(u_{k}(x) \cdot \xi^t_j)| \, \di x \nonumber\\
  &\leq \int_{\Pi_{\xi^t_j}}\bigg(\int_{(U_t + \sqrt{t}e_i)_y^{\xi^t_j}} | D\tau(\hat{u}_y^{\xi^t_j}) | ((s,s+s_t))\, \di s\bigg)\di \HH^{n-1}(y) \nonumber\,.
  \end{align}
By a mollification argument, we have that
\[
\begin{split}
   \int_{\Pi_{\xi^t_j}}  \bigg(\int_{(U_t + \sqrt{t}e_i)_y^{\xi^t_j}}  & | D\tau(\hat{u}_y^{\xi^t_j}) | ((s,s+s_t))\, \di s\bigg)\di \HH^{n-1}(y) 
  \\
  &
  = \int_{\Pi_{\xi^t_j}}\bigg(\int_0^{s_t} | D\tau(\hat{u}_y^{\xi^t_j}) | ((U_t+\sqrt{t}e_i)_y^{\xi^t_j}+\lambda)\, \di \lambda\bigg)\di \HH^{n-1}(y) \,,
  \end{split}
\]
so that we obtain from~\eqref{e:100} that 
\begin{align*}
 \int_{U_t}  |\tau(u_{k}(x + t e_j) \cdot \xi_j^t ) &  - \tau(u_{k}(x - \sqrt{t} e_i) \cdot \xi^t_j)| \, \di x\\
& \leq \int_{\Pi_{\xi^t_j}}\bigg(\int_0^{s_t} | D\tau(\hat{u}_y^{\xi^t_j}) | ((U_t+\sqrt{t}e_i)_y^{\xi^t_j}+\lambda)\, \di \lambda\bigg)\di \HH^{n-1}(y) \\
  &\leq\int_0^{s_t}\bigg( \int_{\Pi_{\xi^t_j}} | D\tau(\hat{u}_y^{\xi^t_j}) | (U_y^{\xi^t_j})\, \di \HH^{n-1}(y) \bigg) \di \lambda \leq \pi s_t  \hat{\mu}_{u_k}(U) \,.
\end{align*}
Analogously,
\[
\int_{U_t} |\tau(u_{k}(x - \sqrt{t} e_i) \cdot e_i ) - \tau(u_{k}(x) \cdot e_i)| \, \di x \leq \pi  \sqrt{t}  \hat{\mu}_{u_k}(U) \,.
\]
 
 Summarizing, we have shown that if $t_h$  is such that $t_h \in(0,  \overline{t}]$ and
 \[
|U| \frac{2\sqrt{t_h}}{c(\delta_h)} \leq \epsilon_h  \qquad \text{and} \qquad \ \pi  s_{t_h}  \hat{\mu}_{u_k}(U) \leq \epsilon_h \,,
 \]
then for every $t \leq t_h$ we have for every $e_j \in \Xi$
 \[
 \int_{U_t} |\tau(u_{k}(x + t e_j) \cdot e_i) - \tau(u_{k}(x)\cdot e_i)| \, \di x \leq \UUU 10 \epsilon_h \EEE \qquad \text{for every $k \geq \phi(h)$}\,.
 \]
 As a consequence, there exists a positive constant~$L= L(n)$ such that
\[
 \int_{U_t} |\tau(u_{k}(x + t \xi) \cdot e_i) - \tau(u_{k}(x)\cdot e_i)| \, \di x \leq L(n)\epsilon_h \qquad \xi \in \mathbb{S}^{n-1}, \ \ k \geq \phi(h), \ \ t \leq t_h.
 \]
Since the index $i$ chosen at the beginning was arbitrary, this means also that if we consider the diffeomorphism $\psi \colon \mathbb{R}^n \to (-\pi/2,\pi/2)^{n}$ defined by $\psi(x):= (\tau(x_1), \dotsc, \tau(x_n))$, then
\[
\int_{U_t} |\psi(u_{k}(x + t\xi)) - \psi(u_{k}(x))| \, \di x \leq L'(n)\epsilon_h, \qquad \xi \in \mathbb{S}^{n-1}, \ \ k \geq \phi(h), \ \ t \leq t_h.
\]

 By Fr\'echet-Kolmogorov Theorem, this last inequality implies that the sequence~$\psi(u_{k})$ is relatively compact in~$L^1(U; \R^{n})$. Hence, we can pass to another subsequence, still denoted by~$\psi(u_{k})$, such that $\psi(u_{k}) \to v$ as $k \to \infty$ strongly in~$L^1(U; \R^{n})$. By eventually passing through another subsequence, we may suppose $\psi(u_{k}(x)) \to v(x)$ a.e.~in~$U$ as $k \to \infty$. As a consequence, there exists a measurable $u \colon U \to \overline{\mathbb{R}}$ such that $u_{k}(x) \to u(x)$ as $k \to \infty$ a.e.~in~$U \setminus \{x \in U: \, v(x) \in \partial (-\frac{\pi}{2},\frac{\pi}{2})^n\}$. Moreover, $|u_{k}(x)| \to +\infty$ if and only if for at least one index~$i$, $u_{k}(x) \cdot e_i \to \pm\infty$ (clearly $\tau(u \cdot e_i) = v_i$) or equivalently if and only if $x \in \{x \in U : \, v(x) \in \partial (-\frac{\pi}{2},\frac{\pi}{2})^n\}$. Thus, we obtain that $u_{k} \to u$ a.e.~in $U \setminus A$ as $k \to \infty$.
 
To show that $A := \{x \in U : \, |u_{k}(x)| \to +\infty  \}$ has finite perimeter the argument follows that in~\cite{cris}. We give a sketch of the proof.  
 
It is easy to check that for $\mathcal{H}^{n-1}$-a.e.~$\xi \in \mathbb{S}^{n-1}$ it holds true
 \begin{equation}
 \label{e:compatibility}
 x \in A  \ \ \text{ if and only if } \ \  \lim_{k \to \infty}\tau(u_{k}(x) \cdot \xi) = \pm \frac{\pi}{2}\, , \ \ \text{ for a.e. }x \in U\,.
 \end{equation}
Now fix $\sigma \geq 1$. First of all using also \eqref{e:compatibility} we can follow a standard measure theoretic argument which shows that we can extract a subsequence, still denoted as $(u_k)_k$, such that for $\mathcal{H}^{n-1}$-a.e. $\xi \in \mathbb{S}^{n-1}$ for $\mathcal{H}^{n-1}$-a.e.~$y \in \Pi^\xi$ it holds true
\begin{equation}
\label{e:convae}
    \tau((\hat{u}_k)_y^\xi) \to v^\xi_y := \begin{cases}
     \tau(\hat{u}_y^\xi) &\text{ on } U^\xi_y \setminus A^\xi_y\\
     \pm \frac{\pi}{2} &\text{ on }A^\xi_y,
    \end{cases}
    \ \ \text{ in }L^1(U^\xi_y) \,.
\end{equation}
Fix $\epsilon >0$. By Fatou Lemma and Remarks~\ref{r:bound} and~\ref{r:sigma} we estimate
\begin{align}
\label{e:lsc1}
&\int_{\Pi_\xi} \liminf_{k\to \infty} \, [\epsilon\,|D(\hat{u}_k)_y^\xi|(U_y^\xi \setminus J_{(\hat{u}_k)_y^\xi}^\sigma) + \mathcal{H}^0(U_y^\xi \cap J_{(\hat{u}_k)_y^\xi}^\sigma)] \, \di \HH^{n-1}(y)\\
&\leq\int_{\Pi_\xi} \liminf_{k\to \infty} \,[\epsilon\,|D(\hat{u}_k)_y^\xi|(U_y^\xi \setminus J_{(\hat{u}_k)_y^\xi}^\sigma) + \mathcal{H}^0(U_y^\xi \cap (J^\sigma_{u_k})_y^\xi)] \, \di \HH^{n-1}(y) \nonumber\\
&\leq \limsup_{k \to \infty}\bigg( \epsilon \,\hat{\mu}^\xi_{u_k}(U) + \epsilon(\sigma-1)\int_{U \cap J^1_{u_k}} \!\!\!\! |\nu_{u_k} \cdot \xi| \, \di \mathcal{H}^{n-1}\bigg) + \liminf_{k \to \infty} \int_{U \cap J^\sigma_{u_k}} \!\!\!\! |\nu_{u_k} \cdot \xi| \, \di \mathcal{H}^{n-1} \nonumber \\
& \leq \epsilon\sup_{k \in \mathbb{N}} \, (1+4n(\sigma-1))\hat{\mu}_{u_k}(U) + \liminf_{k \to \infty} \int_{U \cap J^\sigma_{u_k}}\!\!\!\! |\nu_{u_k} \cdot \xi| \, \di \mathcal{H}^{n-1} < +\infty\,. \nonumber
\end{align}
 For $\mathcal{H}^{n-1}$-a.e.~$y$ we can thus consider a subsequence depending on~$y$ but still denoted by~$(u_k)_k$ such that
 \begin{equation}
 \label{e:slicebound}
 \sup_{k \in \mathbb{N}} \, \epsilon \, |D(\hat{u}_k)_y^\xi|(U_y^\xi \setminus J_{(\hat{u}_k)_y^\xi}^\sigma) + \mathcal{H}^0(U_y^\xi \cap J_{(\hat{u}_k)_y^\xi}^\sigma) < +\infty \,.
  \end{equation}
 
 Now we study the behavior of a sequence of one dimensional functions satisfying~\eqref{e:slicebound}. Let $(a,b) \subset \mathbb{R}$ be a non-empty open interval and suppose that $(f_k)_k$ is a sequence in $BV_{\text{loc}}((a,b))$ satisfying
 \begin{equation}
 \label{e:onedimbound}
 \sup_{k \in \mathbb{N}} |Df_k|((a,b) \setminus J_{f_k}^\sigma) + \mathcal{H}^0(J_{f_k}^\sigma) < \infty.
 \end{equation}
 We write $f_k = f^1_k + f^2_k$ for $f^1_k,f^2_k \colon (a,b) \to \mathbb{R}$ defined as
 \[
 f^1_k(t) :=  Df_k((a,t) \setminus J_{f_k}^\sigma) \ \ \text{and} \ \ f^2_k(t):=f_k(a) + Df_k((a,t) \cap J_{f_k}^\sigma). \]
 We study the convergence of $f^1_k$ and $f^2_k$ separately. 
 
 Inequality \eqref{e:onedimbound} tells us that up to extract a further not relabelled subsequence 
 \begin{equation}
 \label{e:convbv}
 f^1_k \to f^1 \text{ pointwise a.e. for some } f^1 \in BV((a,b)) \text{ as } k \to \infty.
 \end{equation}
 As for~$(f^2_k)_k$, by inequality \eqref{e:onedimbound} we may suppose that, up to extract a further not relabelled subsequence, there exists a finite set $J \subset [a,b]$ such that
 \begin{align}
& \mathcal{H}^0(J) \leq \sup_{k \in \mathbb{N}} \,\mathcal{H}^0(J_{f_k}^{\sigma})\,, \label{e:cond-1}\\
     & J_{f_k}^{\sigma} \to J \qquad \text{in Hausdorff distance as $k \to \infty$}.\label{e:cond-2}
 \end{align}
 Then, \eqref{e:cond-1}--\eqref{e:cond-2} together with the fact that by construction~$f_k^2$ is a piecewise constant function allows us to deduce that any pointwise limit function $f^2$ for $(f_k^2)_k$ must be of the form
  \begin{equation*}
      f^2(t) = \sum_{l=1}^{M} \alpha_l \mathbbm{1}_{(a_l,a_{l+1})}(t) \qquad \text{for $t \in (a,b)$},
  \end{equation*}
  for a suitable $M \leq \mathcal{H}^0(J \cap (a,b)) +1$, for suitable $\alpha_l \in \mathbb{R} \cup \{\pm\infty\}$ with $\alpha_l \neq \alpha_{l+1}$, and for suitable $a_l \in J$ with $a_l < a_{l+1}$ and $a_1=a$, $a_{\mathcal{H}^0(J \cap (a,b)) +2}=b$. Up to extract a further not relabelled subsequence we may suppose $f_k^2 \to f^2$ pointwise a.e.. Now if $\alpha_l \in \{\pm\infty\}$ and $l\neq 1$ and $l\neq \mathcal{H}^0(J \cap (a,b)) +1$, we set 
  \begin{align*}
  & T_{l,k} := \{t \in J^\sigma_{f^2_k} : \, |t-a_l| \leq 1/2 \min_{t_1,t_2 \in J}|t_1 -t_2| \}\,,\\
  &  T_{l+1,k} := \{t \in J^\sigma_{f^2_k} : \, |t-a_{l+1}| \leq 1/2 \min_{t_1,t_2 \in J}|t_1 -t_2| \}\,,
  \end{align*}
  while if $l= 1$ we set
  \[
  T_{l,k} := \{t \in J^\sigma_{f^2_k} : \, |t-a_{l+1}| \leq 1/2 \min_{t_1,t_2 \in J}|t_1 -t_2| \} \,,
  \]
  and if $l= M$  we set
  \[
  T_{l,k} := \{t \in J^\sigma_{f^2_k} :\, |t-a_l| \leq 1/2 \min_{t_1,t_2 \in J}|t_1 -t_2| \} \,.
  \]
  By~\eqref{e:cond-2} we have $T_{l,k} \neq \emptyset$ for every but sufficiently large~$k$ and thanks to the definition of~$T_{l,k}$ any sequence $(t_{l,k})_k$ with $t_{l,k} \in T_{l,k}$ is such that $t_{l,k} \to \alpha_l$ as $k \to \infty$. We claim that for every $l \in \{1,\dotsc,M \}$ there exists one of such sequences~$(t_{l,k})_k$ such that 
  \begin{equation}
  \label{e:notbound}
  \lim_{k \to \infty} |[f^2_k(t_{l,k})]|= + \infty\,. 
  \end{equation}
  Suppose by contradiction that there exists~$l$ and a subsequence~$k_j$ such that
  \[
  \sup_{j \in \mathbb{N}}\, \max_{t \in T_{l,k_{j}}} |[f^2_{k_{j}} (t)]| < + \infty\,.
  \]
 Then, we are in the following situation: we choose one of the endpoints $a_l$ or $a_{l+1}$, for example $a_l$, (in the case $l=1$ we choose $a_{l+1}$ and in the case $l=M$ we choose $a_l$) and the sequence $v_j :=f^2_{k_j} \restr (a_l - \frac{1}{2} \min_{t_1,t_2 \in J}|t_1 -t_2|,a_l + \frac{1}{2} \min_{t_1,t_2 \in J}|t_1 -t_2|)$ satisfies
 \begin{align*}
  & \text{$v_j$ is piecewise constant,}\\
  &   J_{v_j}= T_{ l, k_j }\, \quad \text{and}\quad \text{$J_{v_j} \to a_l$ in Hausdorff distance as $j \to \infty$,}\\
  &  \sup_{j \in \mathbb{N}} \, \mathcal{H}^0(T_{ l, k_j }) < + \infty, \qquad  \sup_{j \in \mathbb{N}} \, \max_{t \in J_{v_j}} \, |[v_j](t)| < + \infty\,.
 \end{align*}
 It is easy to see that the previous conditions are in contradiction with the fact that $f^2 \restr (a_l-1/2 \min_{t_1,t_2 \in J}|t_1 -t_2|,a_l+1/2 \min_{t_1,t_2 \in J}|t_1 -t_2|)$, i.e. the pointwise limit of~$v_j$, is such that~$f^2$ has a non finite jump point at $a_l$. This proves our claim. Our claim implies in particular that, being $(f^1_k)_k$ equibounded, then the sequence~$t_{l, k}$ satisfying~\eqref{e:notbound} is actually contained for every but sufficiently large~$k$ in~$J^\sigma_{f_k}$ (roughly speaking the jumps of~$f^1_k$ cannot compensate a non-bounded sequence of jumps of $f^2_k$). Clearly, being the interval $\{t  : \, |t-a_l| < \frac{1}{2} \min_{t_1,t_2 \in J} |t_1-t_2|\}$ pairwise disjoints for $l \in \{2,\dotsc M \}$ (we are avoiding the end points $a$ and $b$), then we have actually proved the following lower semi-continuity property
 \begin{equation}
     \label{e:onedimlow}
     \mathcal{H}^0(\partial^*\{f=\pm\infty\})=\mathcal{H}^0(\{t \in (a,b) \cap J_f : \, |[f(t)]|=\infty\}) \leq \liminf_{k \to \infty} \mathcal{H}^0(J^\sigma_{f_k})\,,
 \end{equation}
 where $f:= f_1+f_2$. Notice that the set~$J_f$ is well defined since $f$ is the sum of a (bounded) $BV$ function and a piecewise constant function which might assume values~$\pm \infty$, but jumps only at finitely many points.
 
 Having this in mind we can come back to our original problem. Fix $\xi \in \mathbb{S}^{n-1}$ satisfying \eqref{e:convae}. Given $y \in \Pi^\xi$ for which \eqref{e:convae} and \eqref{e:slicebound} hold true we can pass through a not relabelled subsequence (depending on $y$) for which the following liminf
 \[
 \liminf_{k\to \infty}\,  [\epsilon \, |D(\hat{u}_k)_y^\xi|(U_y^\xi \setminus J_{(\hat{u}_k)_y^\xi}^\sigma) + \mathcal{H}^0(U_y^\xi \cap J_{(\hat{u}_k)_y^\xi}^\sigma)]
 \]
 is actually a limit. Passing through a further not relabelled subsequence, we may also suppose that \eqref{e:onedimlow} holds true in each connected component of $U_y^\xi$, i.e.
 \begin{equation*}
     \mathcal{H}^0(\partial^*\{v_y^\xi =\pm\pi/2\})\leq \liminf_{k \to \infty} \, \mathcal{H}^0(J^\sigma_{(\hat{u}_k)_y^\xi})\,.
 \end{equation*}
 Notice that $|v_y^\xi| < \pi/2 $ a.e.~on $U_y^\xi \setminus A_y^\xi$, hence $\{v_y^\xi =\pm\pi/2\}= A_y^\xi$ a.e.~and so $\partial^*\{v_y^\xi =\pm\pi/2\}= \partial^*A_y^\xi$. In particular
 \begin{equation}
 \label{e:sliceA}
    \mathcal{H}^0(\partial^*A_y^\xi)\leq \liminf_{k \to \infty} \mathcal{H}^0(J^\sigma_{(\hat{u}_k)_y^\xi}) 
 \end{equation}
 Therefore, by passing through suitable subsequences, each depending on $y$, when computing the liminf inside the left-hand side integral of \eqref{e:lsc1} and by using \eqref{e:sliceA} we infer
 \begin{align}
     \label{e:lsc2}
\int_{\Pi^\xi} &\mathcal{H}^0(\partial^*A_y^\xi)\, \di \HH^{n-1}(y)
\\
&
\leq \epsilon\, \sup_{k \in \mathbb{N}} \, (1+4n(\sigma-1))\hat{\mu}_{u_k}(U) + \liminf_{k \to \infty} \int_{U \cap J^\sigma_{u_k}} |\nu_{u_k} \cdot \xi| \, \di \mathcal{H}^{n-1}\, \nonumber.
 \end{align}
 The arbitrariness of $\xi$ implies that~\eqref{e:lsc2} holds for $\mathcal{H}^{n-1}$-a.e.~$\xi \in \mathbb{S}^{n-1}$. Hence, we deduce that~$A$ has finite perimeter in~$U$. In addition, by taking the integral on~$\mathbb{S}^{n-1}$ on both sides of~\eqref{e:lsc2} we infer
 \begin{equation*}
         \alpha_{n} \HH^{n-1}( \partial^{*}A)  \leq \epsilon n\omega_n(1+4n(\sigma-1))\sup_{k \in \mathbb{N}} \, \hat{\mu}_{u_k}(U) +\alpha_{n} \liminf_{k \to \infty} \, \mathcal{H}^{n-1}(J^\sigma_{u_k})\,,
 \end{equation*}
  where $\alpha_n:= \int_{\mathbb{S}^{n-1}}|\nu\cdot \xi|$. Moreover, the arbitrariness of $\epsilon >0$ tells us
  \begin{equation*}
      \mathcal{H}^{n-1}(\partial^*A) \leq \liminf_{k \to \infty} \, \mathcal{H}^{n-1}(J^\sigma_{u_k}) \,.
  \end{equation*}
 Finally, by the arbitrariness of $\sigma \geq 1$ and by the fact that $J^{\sigma_1} \subset J^{\sigma_2}$ for $\sigma_1 \geq \sigma_2$ we conclude~\eqref{e:jump-lsc}.
 
 In order to show that $u$ can be extended to the whole of~$U$ as a function in~$GBD(U)$, we define the sequence of $GBD(U)$ functions by 
 \[
 \tilde{u}_k(x):=
 \begin{cases}
 u_k(x) &\text{ if }x \in U \setminus A\,,\\
 0 &\text{ if } x \in A\,.
 \end{cases}
 \]
 Clearly, if we define $v$ as
 \begin{equation}
 \label{e:v-function}
 v(x):=
 \begin{cases}
 u(x) &\text{ if }x \in U \setminus A\,,\\
 0 &\text{ if } x \in A\,.
 \end{cases}
 \end{equation}
then we have $\tilde{u}_k \to v$ a.e.~in~$U$ and 
 \[
 \sup_{k \in\mathbb{N}} \,\hat{\mu}_{\tilde{u}_k}(U) \leq \sup_{k \in \mathbb{N}} \, \hat{\mu}_{u_k}(U)  + \mathcal{H}^{n-1}(\partial^*A) < +\infty \,.
 \]
 Therefore, by using the technique developed in \cite{MR1630504, dal} we can conclude $v \in GBD(U)$.
\end{proof}

 \begin{remark}
Under the additional assumption~\eqref{e:hp-add} with~$u_{k} \in GSBD(U)$, we can obtain the further information $e(u_k) \mathbbm{1}_{U \setminus A } \rightharpoonup e(u)$ in~$L^1(U; \M^{n}_{sym})$ thanks to $e(\tilde{u}_{k}) \rightharpoonup e(u)$ in $L^1(U; \mathbb{M}^{n}_{sym})$ together with the fact $e(u_k) \mathbbm{1}_{U \setminus A } = e(\tilde{u}_k)$ for every $k \in \mathbb{N}$. Moreover,~\eqref{e:onedimlow} can be modified in the following way:
 \begin{equation*}
   \HH^{0} (J_{f}  \cup \partial^*\{f=\pm\infty\} )  \leq \liminf_{k \to \infty} \, \mathcal{H}^0(J_{f_k})\,,
 \end{equation*}
from which it is possible to deduce that
 \[
\mathcal{H}^{n-1}(J_u \cup \partial^* A) \leq  \liminf_{k \to \infty} \, \mathcal{H}^{n-1}(J_{u_k})  \,.
 \]
 \UUU Condition~\eqref{e:hp-add} would also imply that in~\eqref{e:lsc1} we actually control
 \begin{align*}
 \int_{\Pi_\xi} \liminf_{k\to \infty} \, & \Big[ \int_{U_{y}^{\xi}}\epsilon \,\phi(| (\dot{u}_{k})^{\xi}_{y}(t) | )\, \di t   
 + \mathcal{H}^0(U_y^\xi \cap J_{(\hat{u}_k)_y^\xi}) \Big ] \, \di \HH^{n-1}(y)<+\infty\,,
 \end{align*}
 where $(\dot{u}_{k})^{\xi}_{y}$ denotes the absolutely continuous part of~$D(\hat{u}_{k})^{\xi}_{y}$. This in turns allows us to use the well known compactness result for $SBV$ functions in one variable to deduce that the pointwise limit function $f^1$ in \eqref{e:convbv} belongs to $SBV((a,b))$. For this reason, the techniques of~\cite{MR1630504, dal} can be adapted to deduce $v \in GSBD(U)$ (see~\eqref{e:v-function} for the definition of~$v$). The convergence of~$e(u_{k})$ to~$e(u)$ in~$L^{2}(\Om\setminus A; \mathbb{M}^{n}_{sym})$ follows instead by the arguments of~\cite[pp.~10--11]{cris2}.\EEE
%
 \end{remark}

\section*{Acknowledgements}

S.A. acknowledges the support of the OeAD-WTZ project CZ 01/2021 and of the FWF through the project I 5149. E.T. was partially supported by the Austrian Science Fund (FWF) through the project F65.

\bibliographystyle{siam}
\bibliography{A-T_20_comp_bib.bib}

\end{document}